\documentclass{amsart}
\usepackage{amsmath}
\usepackage{amssymb}
\usepackage{amsthm}
\usepackage{enumerate}
\usepackage{xcolor}

\newtheorem{theorem}{Theorem}
\newtheorem{prop}[theorem]{Proposition}
\newtheorem{lemma}[theorem]{Lemma}
\newtheorem{rem}[theorem]{Remark}

\newtheorem{cor}[theorem]{Corollary}

\begin{document}

\title{Four relations on the set of point-hyperplane anti-flags}
\author{Mark Pankov, Antonio Pasini}
\address{Mark Pankov: Faculty of Mathematics and Computer Science, 
University of Warmia and Mazury, S{\l}oneczna 54, 10-710 Olsztyn, Poland}
\email{pankov@matman.uwm.edu.pl}
\address{Antonio Pasini: Department of Information Engineering and Mathematics, University of Siena, Via Roma 56, I-53100 Siena, Italy}
\email{antonio.pasini@unisi.it}

\begin{abstract}
There are precisely four arrangements  of two point-hyperplane anti-flags.
We consider the corresponding relations on the set of such anti-flags and show that 
each of them can be recovered from any other except in one special case. 
If the field consists of two elements, then one of the relations cannot be used to recover each of the remaining three. 
This is related to a bijection between anti-flags and exterior points of the hyperbolic polar space which exists in this case.
\end{abstract}

\maketitle
\section{Introduction}
A point-hyperplane anti-flag of ${\rm PG}(n-1,{\mathbb F})$ is a pair of non-incident point and hyperplane.
There are precisely four arrangements of such two anti-flags. We consider the corresponding relations $\sim_i$, $i\in \{1,2,3,4\}$ on the set of anti-flags
for an arbitrary (not necessarily finite) field  ${\mathbb F}$.

If the characteristic of ${\mathbb F}$ is distinct from two, then every anti-flag can be identified with an involution
$u\in {\rm GL}(n, {\mathbb F})$ such that one of the subspaces ${\rm Ker}(id-u)$, ${\rm Im}(id- u)$ is the point and 
the other is the hyperplane of the anti-flag
(note that $u$ is defined up to the sign $\pm$). 
Two of our relations ($\sim_2$ and $\sim_3$) correspond to the cases when the involutions related to anti-flags  eiher commute or their composition is a transvection. 
These relations  were used in the description of automorphisms of the group ${\rm GL}(n, {\mathbb F})$, see \cite{Die,Pank-book}.

We show that every $\sim_j$ can be recovered from any other $\sim_i$ except in the case when 
$i=1$ and ${\mathbb F}$ is the field of two elements (Theorem \ref{th1}).
As a consequence, we obtain that the automorphism group of the graph $\Gamma_i$ defined by $\sim_i$
is the semidirect product ${\rm P\Gamma  L}(n, {\mathbb F})\rtimes C_2$ except in the case mentioned above.

Describe briefly the exceptional case.
By \cite{IP}, there is a bijection between anti-flags of ${\rm PG}(n-1,2)$
and non-singular points of the hyperbolic form on ${\mathbb F}^{2n}_2$ or, equivalently, exterior points of the polar space ${\mathcal O}^{+}(2n,2)$
(we describe this bijection in Section 4).
Two anti-flags of ${\rm PG}(n-1,2)$ are related by $\sim_1$ if and only if the corresponding non-singular points are joined by a totally non-singular line. 
Therefore, the graph $\Gamma_1$  is the complement of the well-known graph $NO^{+}(2n,2)$ \cite[Subsection 3.1.2]{BvM}.
We show that the polar space ${\mathcal O}^{+}(2n,2)$ can be recovered from $\Gamma_1$ (Theorem \ref{th2});
in particular,  the automorphism group of the graph $\Gamma_1$ is the orthogonal group ${\rm O}^{+}(2n,2)$. 
So, the automorphism groups of the graphs $\Gamma_1$ and $\Gamma_i$, $i\in \{2,3,4\}$ are distinct which implies that $\sim_i$ cannot be recovered from $\sim_1$.

\section{Main objects and results}
Let ${\mathbb F}$ be a field (not necessarily finite) and let $n$ be a natural number not less than three.
Consider the $n$-dimensional vector space $V={\mathbb F}^n$  and the associated projective space ${\rm PG}(n-1,{\mathbb F})$.

A {\it point-hyperplane anti-flag} of ${\rm PG}(n-1,{\mathbb F})$ is a pair $(p,H)$, where $p$ is a point and $H$ is a hyperplane 
such that $p\not\in H$. The set of all such anti-flags will be denoted by ${\mathcal A}$.
We consider the four relations $\sim_i$, $i\in \{1,2,3,4\}$ on ${\mathcal A}$ defined as follows.
For distinct anti-flags $A_j=(p_j,H_j)$, $j\in \{1,2\}$ we write
\begin{itemize}
\item[]$A_1\sim_1 A_2$ if and only if $p_j\in H_{3-j}$ and $p_{3-j}\not\in H_j$ for certain $j\in \{1,2\}$,
\item[]$A_1\sim_2 A_2$ if and only if $p_j\in H_{3-j}$ for each $j\in \{1,2\}$,
\item[]$A_1\sim_3 A_2$ if and only if either $p_1 = p_2$ or $H_1 = H_2$,
\item[]$A_1\sim_4 A_2$ if and only if $p_1\neq p_2$, $H_1\neq H_2$ and $p_j \not\in H_{3-j}$ for each $j\in \{1,2\}$.
\end{itemize} 
It is clear that every pair of anti-flags satisfies just one of the above four conditions.

\begin{rem}\label{rem-inv}{\rm
Suppose that the characteristic of ${\mathbb F}$ is distinct from two. 
For every involution $u\in {\rm GL}(n, {\mathbb F})$ there are subspaces $S_{+}(u),S_{-}(u)\subset V$
such that $V$ is the direct sum of these subspaces and the restriction of $u$ to $S_{+}(u)$ and $S_{-}(u)$ is $id$ and $-id$, respectively.
Then $u$ is said to be  a $(k,n-k)$-{\it involution} if 
the dimension of $S_{+}(u)$ and $S_{-}(u)$ is $k$ and $n-k$, respectively.
For every $k\in \{1,\dots,n-1\}$ all $(k,n-k)$-involutions form a conjugacy class in ${\rm GL}(n, {\mathbb F})$.
Point-hyperplane anti-flags of ${\rm PG}(n-1,{\mathbb F})$ can be identified with $(1,n-1)$-involutions as well as with $(n-1,1)$-involutions.
For $k=1,n-1$ two distinct $(k,n-k)$-involutions commute if and only if the corresponding anti-flags are related by $\sim_2$.
We say that $t\in {\rm GL}(n, {\mathbb F})$ is a {\it transvection} if 
$$t(x)=x+\alpha(x)x_0,$$
where $\alpha:V\to {\mathbb F}$ is a linear functional and $x_0$ is a non-zero vector in the kernel of $\alpha$. 
For $k=1,n-1$ the composition of two distinct $(k,n-k)$-involutions is a transvection if and only if the corresponding anti-flags are related by $\sim_3$.
The relations $\sim_2$ and $\sim_3$ applied to $(1,n-1)$-involutions were used in the description of automorphisms of the group ${\rm GL}(n, {\mathbb F})$.
See \cite{Die} and \cite[Section 3.7]{Pank-book} for more information.
}\end{rem}

\begin{theorem}\label{th1}
The following assertions are fulfilled:
\begin{itemize}
\item[(1)]For any distinct $i,j\in \{1,2,3,4\}$ the relation $\sim_j$ can be recovered from the relation $\sim_i$ 
if $|{\mathbb F}|\ge 3$ or  $i\neq 1$.
\item[(2)] If $|{\mathbb F}|=2$, then each $\sim_j$, $j\in \{2,3,4\}$ cannot be recovered from $\sim_1$.
\end{itemize}
\end{theorem}

For every $i\in \{1,2,3,4\}$ we denote by
$\Gamma_i$ the simple graph whose vertex set is ${\mathcal A}$ and distinct anti-flags $A_1,A_2$
are adjacent vertices of this graph if and only if $A_1\sim_i A_2$. 
In what follows, we will say that anti-flags $A_1,A_2$ are $i$-{\it adjacent} if $A_1\sim_i A_2$. 
As previously remarked, $\Gamma_i\cap\Gamma_j$ is the totally disconnected graph on $\mathcal A$ for every choice of distinct veritices $i, j\in \{1,2,3,4\}$ while 
$\Gamma_1\cup\Gamma_2\cup\Gamma_3\cup\Gamma_4$ is the complete graph on $\mathcal A$.
 
Every semilinear automorphism of $V$ induces a bijective transformation of ${\mathcal A}$
preserving each $\sim_i$ in both directions, i.e. it is an automorphism of the graph $\Gamma_i$.
Similarly, if $u$ is a semilinear isomorphism of $V$ to the dual vector space $V^*$, then the map
$$(p,H)\to (u(H)^0,u(p)^0)$$
(we write $X^{0}$ for the annihilator of a subspace $X$) is an automorphism of $\Gamma_i$.

The first part of Theorem \ref{th1} gives the following.

\begin{cor}\label{cor1}
If $|{\mathbb F}|\ge 3$ or $i\neq 1$, then every automorphism of the graph $\Gamma_i$ is induced by 
a semilinear automorphism of $V$ or a semilinear isomorphism of $V$ to $V^*$.
\end{cor}

\begin{proof}
For the graph $\Gamma_3$ the statement is known, see \cite{Die} or \cite[Section 3.7]{Pank-book}.
Suppose that $i\ne 3$. If $|{\mathbb F}|\ge 3$ or $i\neq 1$, then every automorphism of $\Gamma_i$ is an automorphism of $\Gamma_3$
by the first part of Theorem \ref{th1}.

To make the presentation self-contained we sketched the arguments used in the case when $i=3$. 
Every maximal clique of $\Gamma_3$ consists of all anti-flags with a common point or a common hyperplane. 
If the intersection of two distinct maximal cliques is non-empty, then the cliques are of different types, i.e.
one of them corresponds to a point and the other to a hyperplane.
Also, for any two distinct maximal cliques of the same type there is a maximal clique of the opposite type 
intersecting each of them. 
This implies that every automorphism of $\Gamma_3$ preserves the type of every maximal clique or changes the types of all maximal cliques. 
In the first case, the automorphism induces bijective transformations of the set of points and the set of hyperplanes 
which implies that it can be obtained from a semilinear automorphism of $V$.
In the second, it induces bijections of the set of points and the set of hyperplanes to the set of hyperplanes and the set of points (respectively);
such an automorphism corresponds to a semilinear isomorphism of $V$ to $V^*$.
\end{proof}

\begin{rem}{\rm
If $i\in \{1,2\}$, then for any two pairs of $i$-adjacent anti-flags 
there is a linear automorphism of $V$ sending one of these pairs to the other. 
Similarly, any pair of $3$-adjacent anti-flags can be transferred to any other pair of $3$-adjacent anti-flags
by a linear automorphism of $V$ or a transformation of ${\mathcal A}$ induced by a linear isomorphism of $V$ to $V^{*}$.
So, $\Gamma_i$, $i\in \{1,2,3\}$ is edge-transitive.   
If $(p_1,H_1)$ and $(p_2,H_2)$ are $4$-adjacent anti-flags, then the line $\langle p_1, p_2\rangle$ intersects $H_1\cap H_2$
or it intersects $H_1$ and $H_2$ in distinct points. The second possibility occur if and only if $|{\mathbb F}|\ge 3$. 
Therefore, $\Gamma_4$ is edge-transitive  only in the case when $|{\mathbb F}|=2$.
}\end{rem}

The relation $\sim_1$ for the case when $|{\mathbb F}|=2$ will be considered in Section 4. 
The second part of Theorem \ref{th1} will be obtained as a simple consequence  of the description of automorphisms of 
the graph $\Gamma_1$.

\section{Proof of the first part of Theorem \ref{th1}}

\subsection{Recovering  from $\sim_2$ and $\sim_3$}
If ${\mathcal X}$ is a subset of  the set of anti-flags ${\mathcal A}$, then for every $i\in \{1,2,3,4\}$
denote by ${\mathcal X}^{\sim_i}$  the set of all anti-flags $i$-adjacent to every anti-flag from ${\mathcal X}$.

The following statement provides a simple characterization of the relations $\sim_i$, $i\in \{1,2,4\}$ in terms of the relation $\sim_3$.

\begin{prop}\label{prop-sim3}
For distinct anti-flags $A_1,A_2$ the following assertions are fulfilled:
\begin{enumerate}
\item[{\rm (1)}]  If $A_1\sim_1A_2$, then $|\{A_1, A_2\}^{\sim_3}|=1$.
\item[{\rm (2)}]  If $A_1\sim_2 A_2$, then $|\{A_1, A_2\}^{\sim_3}|=0$.
\item[{\rm(3)}]  If $A_1\sim_4 A_2$, then $|\{A_1, A_2\}^{\sim_3}|=2$.
\end{enumerate}
\end{prop}

\begin{proof}
Suppose that $A_j=(p_j,H_j)$, $j\in \{1,2\}$ are $i$-adjacent for $i\ne 3$.
Then $p_1\ne p_2$ and $H_1\ne H_2$. If an anti-flag $(p,H)$ is $3$-adjacent to  both $A_1,A_2$, then $\{p, H\} = \{p_j, H_{3-j}\}$ for certain $j\in \{1,2\}$.  

If  $A_1\sim_4 A_2$, then $p_j\not\in H_{3-j}$ for each $j$ and
an anti-flag $3$-adjacent to both $A_1,A_2$ is $(p_j,H_{3-j})$ with $j\in \{1,2\}$.

If $A_1\sim_1 A_2$, then  $p_j\in H_{3-j}$ and $p_{3-j}\not\in H_j$ for certain $j\in \{1,2\}$.
In this case,  only the anti-flag $(p_{3-j},H_j)$ is as required.

If $A_1\sim_2 A_2$, then $p_j\in H_{3-j}$ for each $j$ and 
there is no anti-flag $3$-adjacent to both $A_1,A_2$.
\end{proof}

By Mackey \cite{Mackey} (see also \cite[Lemma 3.10]{Pank-book}, also Cohen et al. \cite[Lemma 2.6]{CCG}), the relation $\sim_3$ can be characterized in terms of the relation $\sim_2$ as follows. 

\begin{prop}\label{prop-sim2}
Distinct anti-flags $A_1,A_2$ are $3$-adjacent if and only if for any distinct anti-flags $A'_1,A'_2\in (\{A_1,A_2\}^{\sim_2})^{\sim_2}$
we have 
$$\{A'_1,A'_2\}^{\sim_2})^{\sim_2}=(\{A_1,A_2\}^{\sim_2})^{\sim_2}.$$
\end{prop}

Combining Propositions \ref{prop-sim3} and \ref{prop-sim2} we obtain characterizations of $\sim_1$ and $\sim_4$ in terms of $\sim_2$. 
So, for any distinct $i\in \{2,3\}$ and $j\in \{1,2,3,4\}$ the relation $\sim_j$ can be recovered from $\sim_i$.

\subsection{Recovering from $\sim_4$}
Let ${\mathfrak X}$ be the family of all subsets $\{A_1,A_2\}^{\sim_4}$, where $A_1,A_2$ are $i$-adjacent anti-flags and $i\ne 4$.
We characterize the relations $\sim_i$, $i\in \{1,2,3\}$ in terms of the poset $({\mathfrak X},\subset)$ for almost all cases.

\begin{prop}\label{prop-sim4}
Suppose that $|{\mathbb F}|\ge 3$ or $n\ge 4$.
Then the following assertions are fulfilled: 
\begin{enumerate}
\item[{\rm (1)}] If $A_1\sim_1 A_2$, then $\{A_1,A_2\}^{\sim_4}$ is a minimal and non-maximal element of $({\mathfrak X},\subset)$.
\item[{\rm (2)}] If $A_1\sim_2 A_2$, then $\{A_1,A_2\}^{\sim_4}$ is a minimal and maximal element of $({\mathfrak X},\subset)$.
\item[{\rm (3)}] If $A_1\sim_3 A_2$, then $\{A_1,A_2\}^{\sim_4}$ is a non-minimal and maximal element of $({\mathfrak X},\subset)$.
\end{enumerate}
\end{prop}

The case when $|{\mathbb F}|=2$ and $n=3$ will be considered at the end of this subsection. 
To prove Proposition \ref{prop-sim4} we use the following. 

\begin{lemma}\label{lemma-3hyp}
For mutually distinct hyperplanes $H_1,H_2,H_3$ the following assertions are fulfilled: 
\begin{enumerate}
\item[{\rm (1)}] If $H_1\cap H_2\not\subset H_3$ and $|{\mathbb F}|\ge 3$ or $n\ge 4$, 
then there are at least two points outside of $H_1\cup H_2\cup H_3$.
\item[{\rm (2)}] 
If $H_1\cap H_2\subset H_3$ and $|{\mathbb F}|\ge 3$, then
there are more than two points outside of $H_1\cup H_2\cup H_3$.
\end{enumerate}
\end{lemma}

\begin{proof}
If ${\mathbb F}$ is infinite, then there are infinitely many points outside of  $H_1\cup H_2\cup H_3$  in the both cases. 
Suppose that $|{\mathbb F}|=q$.  If $H_1\cap H_2\not\subset H_3$, then 
$$|H_1\cup H_2\cup H_3|=|H_1|+|H_2|+|H_3|-|H_1\cap H_2|-|H_1\cap H_3|-|H_2\cap H_3|+|H_1\cap H_2\cap H_3|=$$
$$\frac{3(q^{n-1}-1)-3(q^{n-2}-1)+q^{n-3}-1}{q-1}=\frac{(3q^2-3q+1)q^{n-3}-1}{q-1}$$
which means that there are precisely 
$$\frac{q^n-1}{q-1}-|H_1\cup H_2\cup H_3|=\frac{(q^3-3q^2+3q-1)q^{n-3}}{q-1}=(q-1)^2q^{n-3}$$ 
points outside of $H_1\cup H_2\cup H_3$.
This number is not less than $2$ if $q\ge 3$ or $n\ge 4$. 
If $H_1\cap H_2\subset H_3$, then 
$$|H_1\cup H_2\cup H_3|=\frac{3(q^{n-1}-1)-2(q^{n-2}-1)}{q-1}$$
and we obtain that 
$$\frac{q^n-1}{q-1}-|H_1\cup H_2\cup H_3|=\frac{(q^2-3q+2)q^{n-2}}{q-1}=(q-2)q^{n-2}.$$
This number is greater than $2$ if $q\ge 3$.
\end{proof}

\begin{lemma}\label{lemma-tex4}
Suppose that $|{\mathbb F}|\ge 3$ or $n\ge 4$. 
If anti-flags ${\mathcal A}_j=(P_j,H_j)$, $j\in\{1,2\}$ are $i$-adjacent  for $i\ne 4$,
then the following assertions are fulfilled:
\begin{enumerate}
\item[{\rm(1)}] For any point $p$ distinct from $p_1,p_2$ there is an anti-flag $(p',H')\in \{A_1,A_2\}^{\sim_4}$ such that $p\in H'$.
\item[{\rm (2)}] For any  hyperplane $H$ distinct from $H_1,H_2$
there is an anti-flag $(p'',H'')\in \{A_1,A_2\}^{\sim_4}$ such that $p''\in H$.
\end{enumerate}
\end{lemma}

\begin{rem}{\rm
If $|{\mathbb F}|\ge 3$ or $n\ge 4$, then Lemma \ref{lemma-tex4} shows that 
$$(\{A_1,A_2\}^{\sim_4})^{\sim_4}=\{A_1,A_2\}$$
for any $i$-adjacent anti-flags $A_1,A_2$ with $i\ne 4$. 
In the case when $|{\mathbb F}|=2$ and $n=3$, this fails, see Proposition \ref{prop-sim4F2}.
}\end{rem}

\begin{proof}[Proof of Lemma \ref{lemma-tex4}]
The statement (2) follows from the statement (1) by duality.
Also, it is sufficient to prove the statement (1) for the case when $A_1,A_2$ are $i$-adjacent with $i\in \{1,2\}$.

Indeed, if $p_1=p_2$, then we take any point $p'_2\in H_1\setminus (H_2\cup\{p\})$.
The anti-flags $A_1=(p_1,H_1)$ and $A'_2=(p'_2,H_2)$ are $1$-adjacent.
If $(p',H')$ belongs to $\{A_1,A'_2\}^{\sim_4}$, then $p'\not\in H_2$ and $p_1=p_2\not\in H'$  which means that $(p',H')$ 
is $4$-adjacent to $A_2$.

Similarly, if $H_1=H_2$, then we consider $A'_2=(p_2, H'_2)$, where $H'_2$ is a hyperplane which contains $p_1$ and does not contain $p_2$.
As above, $A_1$ and $A'_2$ are $1$-adjacent.
If $(p',H')$ belongs to $\{A_1,A'_2\}^{\sim_4}$,
then $p'\not\in H_1=H_2$ and $p_2\not\in H'$, i.e.  $(p',H')$ is $4$-adjacent to $A_2$.

From this moment, we suppose that the anti-flags $A_1,A_2$ are $i$-adjacent for $i\in \{1,2\}$.
Then $p_1\ne p_2$, $H_1\ne H_2$ and at least one of $p_j$ is contained in $H_{3-j}$. 
Without loss of generality we assume that $p_2\in H_1$. 

Consider the case when $|{\mathbb F}|\ge 3$. 

Suppose that $p$ is not on the line $\langle p_1,p_2\rangle$. 
By our assumption, $H_1$ intersects this line  precisely in $p_2$. 
The intersecting point of $H_2$ and $\langle p_1,p_2\rangle$ may be distinct from $p_1$. 
Since there are at least four points on $\langle p_1,p_2\rangle$,
the line contains a point $\tilde{p}\not\in H_1\cup H_2$ distinct from $p_1,p_2$.  
We take any hyperplane $H'$ passing through $p$ and intersecting $\langle p_1,p_2\rangle$ precisely  in $\tilde{p}$. 

Suppose now that $p\in \langle p_1,p_2\rangle$. Then $p\not\in H_1$ (since $p_2\in H_1$ and $\langle p_1,p_2\rangle\not\subset H_1$).
If $p$ is the intersecting point of  $H_2$ and $\langle p_1,p_2\rangle$, 
then  there is a hyperplane $H'$ intersecting $\langle p_1,p_2\rangle$ precisely  in $p$ and distinct from $H_2$.
If $p\not\in H_2$, then we take any hyperplane $H'$ which  intersects  $\langle p_1,p_2\rangle$ precisely in $p$.

In each of these cases,  we have $H'\ne H_1,H_2$, $p\in H'$ and $p_1,p_2\not\in H'$. 
By Lemma \ref{lemma-3hyp}, there are at least two points outside of $H_1\cup H_2\cup H'$
which implies the existence of a point $p'\not\in H_1\cup H_2\cup H'$ distinct from $p_1$. 
The anti-flag $(p',H')$ is $4$-adjacent to both $A_1,A_2$.

Consider the case when $|{\mathbb F}|=2$ and $n\ge 4$. 

%Note that the line $\langle p_1,p_2\rangle$ does not intersect $H_1\cap H_2$
%(indeed, if this line intersects $H_1\cap H_2$, then, since $p_2\in H_1\setminus H_2$, it is contained in $H_1$ which is impossible). 
If $p\in \langle p_1,p_2\rangle$, then $p\not\in H_1$. 
We choose a hyperplane $S$ of $H_1$ distinct from $H_1\cap H_2$ and such that $p_2\not\in S$. 
The hyperplane $H'=\langle p, S\rangle$ does not contain $H_1\cap H_2$
and intersects $\langle p_1,p_2\rangle$ precisely in $p$.

Suppose that $p\not\in \langle p_1,p_2\rangle$. 
Let $\ell$ be the line joining $p$ with the point  of $\langle p_1,p_2\rangle$ distinct from $p_1,p_2$
and let $\tilde{p}$ be the intersecting point of $H_1$ and $\ell$. 
The points $p$ and $\tilde{p}$ are not necessarily distinct.
On the other hand, $\tilde{p}\ne p_2$ (otherwise, $\langle p_1,p_2\rangle=\ell$ and $p\in \langle p_1,p_2\rangle$ which contradicts our assumption).
Since $n\ge 4$, hyperplanes are not lines. Therefore, there is a hyperplane $S$ of $H_1$ distinct from $H_1\cap H_2$ 
and such that $\tilde{p}\in S$, $p_2\not\in S$. 
The hyperplane $H'=\langle \ell, S\rangle $ contains $p$,  does not contain $H_1\cap H_2$
and intersects $\langle p_1,p_2\rangle$ in the point distinct from $p_1,p_2$.

In each of these cases,  we have $H_1\cap H_2\not\subset H'$, $p\in H'$ and $p_1,p_2\not\in H'$. 
Lemma \ref{lemma-3hyp} implies the existence of a point $p'\not\in H_1\cup H_2\cup H'$ distinct from $p_1$. 
The anti-flag $(p',H')$ is as required.  
\end{proof}

\begin{proof}[Proof of Proposition \ref{prop-sim4}]
Let $A_j=(p_j,H_j)$, $j\in\{1,2\}$ be $i$-adjacent anti-flags 
and let $A'_j=(p'_j,H'_j)$, $j\in\{1,2\}$ be $i'$-adjacent anti-flags 
such that  each of $i,i'$ is distinct from $4$.
Show that the inclusions
\begin{equation}\label{eq-4.1}
\{A_1,A_2\}^{\sim_4}\subset \{A'_1,A'_2\}^{\sim_4}
\end{equation}
and
\begin{equation}\label{eq-4.2}
\{p'_1,p'_2\}\subset \{p_1,p_2\}\;\mbox{ and }\;\{H'_1,H'_2\}\subset \{H_1,H_2\}
\end{equation}
are equivalent.
The implication $(2)\Rightarrow (1)$ is obvious.
Conversely,  if $p'_j$ is distinct from $p_1,p_2$, then, by Lemma \ref{lemma-tex4}, 
there is $(p,H)\in \{A_1,A_2\}^{\sim_4}$ such that $p'_j\in H$ which means that $(p,H)\not\in\{A'_1,A'_2\}^{\sim_4}$. 
Similarly, if $H'_j$ is distinct  from $H_1,H_2$, then there exists $(p,H)\in \{A_1,A_2\}^{\sim_4}$ such that $p\in H'_j$
and $(p,H)\not\in\{A'_1,A'_2\}^{\sim_4}$ again. 
So, $(1)\Rightarrow (2)$.

We prove the statement in several steps.

(I). Suppose that each of $i,i'$ is distinct from $3$.
Then 
$$p_1\ne p_2,\;p'_1\ne p'_2,\; H_1\ne H_2,\; H'_1\ne H'_2.$$
The inclusion \eqref{eq-4.2} implies that 
$$\{p'_1,p'_2\}=\{p_1,p_2\}\;\mbox{ and }\;\{H'_1,H'_2\}=\{H_1,H_2\}.$$
Since $p_j\in H_{3-i}$ for at least one $j\in\{1,2\}$, the inclusion \eqref{eq-4.1}  holds if and only if 
$\{A_1,A_2\}=\{A'_1, A'_2\}$.

(II). For $i=3$ we have $p_1=p_2$ or $H_1=H_2$ which means that \eqref{eq-4.2} holds if and only if  $\{A_1,A_2\}=\{A'_1, A'_2\}$.
Since the inclusions \eqref{eq-4.1} and \eqref{eq-4.2} are equivalent, $\{A_1,A_2\}^{\sim_4}$ is a maximal element of the poset $({\mathfrak X},\subset)$.

(III). Consider the case when $i\in \{1,2\}$ and $i'=3$. 
Since $p'_{j}\not\in H'_{3-j}$ for each $j\in \{1,2\}$, the inclusion \eqref{eq-4.2}  guarantees that $p_j\not\in H_{3-j}$ for a certain $j\in \{1,2\}$
and, consequently, $i=1$. So, the inclusion \eqref{eq-4.1} is impossible for $i=2$.
This observation and the previous steps show that
$\{A_1,A_2\}^{\sim_4}$ is a minimal and maximal element of the poset $({\mathfrak X},\subset)$ if $i=2$. 

Suppose that $i=1$ ($i'=3$ as above). Without loss of generality we assume that $p_1\not\in H_2$ and $p_2\in H_1$. 
Then the inclusion \eqref{eq-4.1} holds if and only if
$\{A'_1,A'_2\}$ is 
$$\{A_1=(p_1,H_1), (p_1,H_2)\}\;\mbox{ or }\;\{(p_1,H_2),A_2=(p_2,H_2)\}.$$
By the step (II), the inclusion reverse to \eqref{eq-4.1}  fails for each of these cases.
Therefore, $\{A_1,A_2\}^{\sim_4}$ is a minimal and non-maximal element of the poset $({\mathfrak X},\subset)$ if $i=1$.

(IV). If $A'_1,A'_2$ are $3$-adjacent, then $\{A'_1,A'_2\}^{\sim_4}$ is a maximal element of the poset $({\mathfrak X},\subset)$.
We need to show that this element is non-minimal, i.e. the inclusion \eqref{eq-4.1} holds for some $1$-adjacent $A_1,A_2$.
If $p'_1\ne p'_2$ and $H'_1=H'_2$, then we take any hyperplane $H_2$ which contains $p'_1$ and does not contain $p'_2$ 
(such a hyperplane is distinct from $H'_1=H'_2$) and put 
$$A_1=A'_1=(p'_1, H'_1),\; A_2=(p'_2,H_2).$$
In the case when $p'_1=p'_2$ and $H'_1\ne H'_2$, there is a point $p_2\in H'_1\setminus H'_2$ 
(such a point is distinct from $p'_1=p'_2$) and 
$$A_1=A'_1=(p'_1, H'_1),\; A_2=(p_2,H'_2)$$
are as required.
\end{proof}

Proposition \ref{prop-sim4} does not cover the case when $|{\mathbb F}|=2$ and $n=3$.
In this case, we characterize $\sim_2$  in terms of $\sim_4$.
Then $\sim_1$ and $\sim_3$  can be characterized in terms of $\sim_4$ by Propositions \ref{prop-sim3} and \ref{prop-sim2}.

\begin{prop}\label{prop-sim4F2}
Suppose that $|{\mathbb F}|=2$ and $n=3$. Then anti-flags $A_1,A_2$ are $2$-adjacent if and only if 
$\{A_1,A_2\}^{\sim_4}$ is non-empty.
\end{prop}

\begin{proof}
Since $n=3$, hyperplanes are lines. 
Let $A_j=(p_j,\ell_j)$, $j\in \{1,2\}$ be distinct anti-flags.

Suppose that $A_1,A_2$ are $2$-adjacent. Then $p_j\in \ell_{3-j}$ for each $j\in \{1,2\}$. 
There is a unique line $\ell$ which  intersects $\ell_1,\ell_2$ in two distinct points different from $p_1,p_2$
and there is a unique point $p\not\in \ell\cup \ell_1\cup \ell_2$. 
The anti-flag $(p,\ell)$ is the unique element of $\{A_1,A_2\}^{\sim_4}$. 

In the remaining cases (if $A_1,A_2$ are $3$-adjacent, then we assume that $\ell_1\ne \ell_2$ by duality), 
there is at most one point outside of $\ell_1\cup \ell_2\cup\{p_1,p_2\}$ which means that
$\{A_1,A_2\}^{\sim_4}$ is empty.
\end{proof}

\subsection{Recovering from $\sim_1$. The case when $|{\mathbb F}|\ge 4$}
Let ${\mathcal C}$ be a coclique  of the graph $\Gamma_1$, i.e. a subset of ${\mathcal A}$, where any two distinct anti-flags are not $1$-adjacent. 
We say that ${\mathcal C}$ is {\it linear} if all anti-flags from ${\mathcal C}$ have a common hyperplane 
and there is a line containing the point of every anti-flag from ${\mathcal C}$. 
The coclique is said to be {\it dually linear} if all anti-flags from ${\mathcal C}$ have a common point and 
there is a line of the dual projective space containing the hyperplane of every anti-flag from ${\mathcal C}$. 
Any two distinct anti-flags in a linear or dually linear coclique are $3$-adjacent.
Every bijective transformation of ${\mathcal A}$ induced by a semilinear automorphism of $V$ to $V^*$
sends linear cocliques to dually linear cocliques and vice versa.

\begin{lemma}\label{lemma-1.1}
If ${\mathcal C}$ is a linear or dually linear coclique of $\Gamma_1$, then for every anti-flag $A\not\in {\mathcal C}$ one of the following possibilities is realized:
\begin{enumerate}
\item[$(1)$] $A^{\sim_1}\cap {\mathcal C}$ is empty,
\item[$(2)$] $A^{\sim_1}\cap {\mathcal C}$ is a singleton,
\item[$(3)$] $A^{\sim_1}\cap {\mathcal C}$ contains all anti-flags of ${\mathcal C}$ except one,
\item[$(4)$] ${\mathcal C}\subset A^{\sim_1}$.
\end{enumerate}
\end{lemma}

\begin{proof}
We can assume that ${\mathcal C}$ is linear by duality.
Let $A=(p,H)\in {\mathcal A}\setminus {\mathcal C}$ and let $\ell,H'$ be the line and hyperplane corresponding to 
the linear coclique ${\mathcal C}$.

Suppose first that $p \in H'$. If $\ell\subset H$, then $A$ is $2$-adjacent to every anti-flag from ${\mathcal C}$
and, consequently, $A^{\sim_1}\cap {\mathcal C}=\emptyset$.
Assume that $\ell\not\subset H$ and $p'$ is the intersecting point of $\ell$ and $H$. If the points of anti-flags of ${\mathcal C}$ are distinct from $p'$, then $A$ is $1$-adjacent to all anti-flags of ${\mathcal C}$,
i.e. ${\mathcal C}\subset A^{\sim_1}$.
Otherwise,  the anti-flag $(p',H')$ belongs to ${\mathcal C}$ and $A^{\sim_1}\cap {\mathcal C}$ contains all anti-flags of ${\mathcal C}$ except $(p',H')$.

Suppose now that $p\not\in H'$.
If $\ell\subset H$, then $A$ is $1$-adjacent to all anti-flags from ${\mathcal C}$ and we obtain that ${\mathcal C}\subset A^{\sim_1}$.
Assume that $\ell\not\subset H$ and $p'$ is the intersecting point of $\ell$ and $H$.
Suppose first that $p'\in H'$. If $p\not\in \ell$, then all elements of ${\mathcal C}$ are $4$-adjacent to $A$ while if $p\in \ell$ then $(p,H')\sim_3 A$ and all remaining elements of 
$\mathcal C$ are $4$-adjacent to $A$. 
In both cases $A^{\sim_1}\cap{\mathcal C} = \emptyset$. 
On the other hand, let $p'\not\in H'$. Then $(p',H')\sim_1 A$. If $p\not\in \ell$, then all anti-flags in ${\mathcal C}\setminus\{(p',H')\}$ are $4$-adjacent to $A$ while if $p\in \ell$ then 
$(p,H')\sim_3 A$ and all elements of ${\mathcal C}\setminus\{(p',H'), (p,H')\}$ are $4$-adjacent to $A$. 
In both cases $A^{\sim_1}\cap{\mathcal C} = \{(p',H')\}$.
\end{proof}

We say that a coclique ${\mathcal C}$ of $\Gamma_1$ satisfies the property (L) if 
for every anti-flag $A\not\in {\mathcal C}$ one of the possibilities described in Lemma \ref{lemma-1.1} is realized.
This property holds for every coclique of $\Gamma_1$ containing less than four elements. 
Note that linear and dually linear cocliques of $\Gamma_1$ containing four elements exist only when $|{\mathbb F}|\ge 4$.

\begin{prop}\label{prop-sim1}
Suppose that $|{\mathbb F}|\ge 4$. A $4$-element coclique of $\Gamma_1$ is linear or dually linear if and only if it satisfies the property {\rm (L)}. 
\end{prop}

If $|{\mathbb F}|\ge 4$, then any two $3$-adjacent anti-flags belong to a linear or dually linear  $4$-element  coclique of $\Gamma_1$
and Proposition \ref{prop-sim1} implies the following characterization of the relation $\sim_3$ in terms of the relation $\sim_1$. 

\begin{cor}\label{cor-sim1}
In the case when $|{\mathbb F}|\ge 4$, two distinct anti-flags are $3$-adjacent if and only if they belong to 
a $4$-element coclique of $\Gamma_1$ satisfying the property {\rm (L)}. 
\end{cor}

Therefore, if $|{\mathbb F}|\ge 4$, then
$\sim_2$ and $\sim_4$ can be characterized in terms of $\sim_1$ by Proposition \ref{prop-sim3}.

To prove Proposition \ref{prop-sim1} we need the following.

\begin{lemma}\label{lemma-F4}
If $|{\mathbb F}|\ge 4$, then for any set formed by five points of ${\rm PG}(n-1,{\mathbb F})$
and any point from this set there is a hyperplane of ${\rm PG}(n-1,{\mathbb F})$ intersecting this set precisely in this point.
\end{lemma}

\begin{proof}
Let $\{p_1,\dots,p_5\}$ be a set of points. Show that there is a hyperplane intersecting this set precisely in $p_5$.
Denote by $S$ the subspace of  ${\rm PG}(n-1,{\mathbb F})$ spanned by this set.
If $S$ is a line, then any hyperplane intersecting this line precisely in $p_5$ is as required. 
If $S$ is a plane, then it contains at least five lines passing through $p_5$ and at least one of them does not intersect
the set $\{p_1,p_2,p_3,p_4\}$; we take any hyperplane which intersects $S$ precisely in this line.

Suppose that $S$ is a solid. There is a plane $S'\subset S$ which does not contain $p_5$. 
For every $i\in \{1,2,3,4\}$ denote by $p'_i$ the  intersecting point of $S'$ and the line $\langle p_i, p_5\rangle$. 
As above, there is a line $\ell\subset S'$ which does not intersects the set $\{p'_1,p'_2,p'_3,p'_4\}$. 
We take any hyperplane intersecting  $S$ precisely in the plane spanned by $\ell$ and $p_5$.

Similarly, if $S$ is $4$-dimensional, then we choose a solid $S'\subset S$ which does not contain $p_5$
and for every $i\in \{1,2,3,4\}$ denote by $p'_i$ the  intersecting point of $S'$ and the line $\langle p_i, p_5\rangle$.
As in the previous case, we establish the existence of  a plane $S''\subset S'$ which does not intersects the set $\{p'_1,p'_2,p'_3,p'_4\}$. 
Any hyperplane intersecting  $S$ precisely in the solid spanned by $S''$ and $p_5$ is as required.
\end{proof}

\begin{proof}[Proof of Proposition \ref{prop-sim1}]
By Lemma \ref{lemma-1.1}, the property (L) holds for every linear or dually linear coclique of $\Gamma_1$.
Let ${\mathcal C} = \{(p_j,H_j)\}_{j=1}^4$ be a $4$-element  coclique of $\Gamma_1$ satisfying (L).
We need to show that it is linear or dually linear.

Suppose  that the points $p_1, p_2, p_3, p_4$ as well as the hyperplanes $H_1, H_2, H_3, H_4$ are mutually distinct.
We consider $H_1\cap H_2, H_1\cap H_3, H_1\cap H_4$ as hyperplanes of $H_1$
and choose a point 
$$
p\in H_1\setminus(H_2\cup H_3\cup H_4).
$$
It is possible that $p\in \{p_2,p_3,p_4\}$, but we can assume that $p\ne p_2$ without loss of generality. 
Lemma \ref{lemma-F4} implies the existence of  a hyperplane $H$ such that 
$$H\cap\{p_1, p_2, p_3, p_4,p\} = \{p_2\}.$$
The anti-flag $(p,H)$ is $1$-adjacent to $(p_i, H_i)$ if and only if $i\in \{1,2\}$.  
Then $(p,H)\not\in {\mathcal C}$ (since ${\mathcal C}$ is a coclique of $\Gamma_1$) and (L) fails, a contradiction. 

So, there are distinct $i,j\in\{1,2,3,4\}$ such that $p_i=p_j$ or $H_i=H_j$.
We can assume that $|\{H_1,H_2,H_3,H_4\}|\le 3$ by duality.

Consider the case when  $|\{H_1,H_2,H_3,H_4\}|=3$, say $H_1 = H_2 =:H'$ and $H', H_3, H_4$ are mutually distinct.
By Lemma \ref{lemma-F4}, there is a hyperplane $H$ such that 
$$H\cap\{p_1, p_2, p_3, p_4,p'\} = \{p'\},$$
where $p'$ is a point outside of $H'$.
Then $H'\ne H$. As above, we choose a point 
$$p\in H'\setminus (H_3\cup H_4\cup H).$$
The anti-flag $(p,H)$ is $1$-adjacent to $(p_i,H_i)$ if and only if $i\in \{1,2\}$ which contradicts (L).
Therefore,  $|\{H_1,H_2,H_3,H_4\}|\le 2$.

Suppose now that $|\{H_1,H_2,H_3,H_4\}|=2$, say $H_1 = H_2 = H_3 =:H'$ and $H_4\neq H'$. 
Lemma \ref{lemma-F4} implies the existence of  a hyperplane $H$ such that 
$$H\cap \{p_1, p_2, p_3, p_4,p'\} = \{p_3\},$$
where $p'\in H'$. Then $H'\ne H$ and there is a point 
$$p\in H'\setminus(H_4\cup H).$$ 
The anti-flag $(p,H)$ is $1$-adjacent to $(p_i,H_i)$ if and only if $i\in \{1,2\}$
which contradicts (L) again.

We obtain that $H_1 = H_2 = H_3 = H_4 =:H'$. If no line contains $\{p_1, p_2, p_3, p_4\}$, 
then for some distinct $i,j\in \{1,2,3,4\}$ the line $\langle p_i, p_j\rangle$ does not contain the remaining two $p_t$
and it is easy to see that there is a hyperplane $H$ intersecting $\{p_1, p_2, p_3, p_4\}$ precisely in two points.
Let $$H\cap\{p_1, p_2, p_3, p_4\} =\{p_3, p_4\}.$$
Since $H'\ne H$, there is $p\in H'\setminus H$.
The anti-flag $(p,H)$ is $1$-adjacent to $(p_i, H')$ if and only if $i\in \{1,2\}$  in contrast to our assumption. 
Therefore, $\{p_1, p_2, p_3, p_4\}$ is contained in a line and the coclique ${\mathcal C}$ is linear.
\end{proof}

\subsection{Recovering from $\sim_1$. The case when ${\mathbb F}$ is finite}
The case of $|\mathbb{F}| = 3$ remains to be considered. However, the arguments to be used to fix this case work as well for any finite field. So, throughout this subsection we assume that $\mathbb{F}$ is finite, $|\mathbb{F}| = 2$ being allowed. 
Thus, we shall be able to show why when $|\mathbb{F}| = 2$ the relations $\sim_i$ with $i > 1$ cannot be recovered from $\sim_1$.

\begin{prop}\label{prop-finite1}
Suppose that $|{\mathbb F}|=q$.
Let $A_i=(p_i, H_i)$, $i\in \{1,2\}$ be distinct anti-flags such that $A_1\not\sim_1 A_2$.
Then the following assertions are fulfilled:
\begin{enumerate}
\item[{\rm (1)}] If $A_1\sim_2 A_2$, then $|\{A_1,A_2\}^{\sim_1}| = 4(q^{n-2}-1)(q-1)q^{n-3}$.
\item[{\rm (2)}] If $A_1\sim_3 A_2$, then $|\{A_1,A_2\}^{\sim_1}| =q^{2n-3} - 2q^{n-2}$.
\item[{\rm(3)}] If $A_1\sim_4 A_2$ and the line joining  $p_1$ and $p_2$ intersects $H_1\cap H_2$ in a point, then 
$|\{A_1,A_2\}^{\sim_1}|= 4q^{n-2}(q^{n-2}-q^{n-3})-2q^{n-2}$.
\item[{\rm (4)}] If $A_1\sim_4 A_2$  and the line joining  $p_1$ and $p_2$ does not intersect $H_1\cap H_2$, then 
$|\{A_1,A_2\}^{\sim_1}|= 4(q^{n-2}-1)(q-1)q^{n-3} + 2q^{n-2}$.
\end{enumerate} 
\end{prop} 

For $q=2$ the fourth possibility of Proposition \ref{prop-finite1} cannot occur and  the numbers (1)-(3) are mutually equal.
If $q\ge 3$, then all above numbers are mutually different
which characterizes each of the relations $\sim_i$, $i\in \{2,3,4\}$ in terms of $\sim_1$.

\begin{proof}[Proof of Proposition \ref{prop-finite1}]
The complement of a hyperplane of ${\rm PG}(n-1,q)$ consists of $q^{n-1}$ points and there are precisely $(q-1)q^{n-2}$ points outside of 
the union of two distinct hyperplanes of ${\rm PG}(n-1,q)$. 
Let $H'_1,H'_2,H'_3$ be mutually distinct hyperplanes of ${\rm PG}(n-1,q)$. 
If $H'_1\cap H'_2\not\subset H'_3$, then there are precisely $(q-1)^2 q^{n-3}$ points outside of $H'_1\cup H'_2\cup H'_3$. 
In the case when $H'_1\cap H'_2\subset H'_3$, the number of such points is $(q-2)q^{n-2}$.

(1). 
Suppose that $A_1\sim_2 A_2$. Then $p_i\in H_{3-i}$ for $i\in \{1,2\}$.
If an anti-flag $(p,H)$ belongs to $\{A_1,A_2\}^{\sim_1}$, then one of the following possibilities is realized:
\begin{enumerate}
\item[$\bullet$] $p\in H_1\cap H_2$ and $H$ intersects $\langle p_1, p_2\rangle$ in a point distinct from $p_1$ and $p_2$;
\item[$\bullet$] $p\not\in H_1\cup H_2\cup \langle p_1, p_2\rangle$ and $\langle p_1, p_2\rangle \subset H$;
\item[$\bullet$] for one of $i\in \{1,2\}$ we have $p\in H_i\setminus (H_{3-i}\cup \{p_{3-i}\})$, $p_i\not\in H$ and $p_{3-i}\in H$.
\end{enumerate}
There are precisely 
$$\frac{q^{n-2}-1}{q-1}\cdot (q-1)^2q^{n-3}=(q^{n-2}-1)(q-1)q^{n-3}$$
anti-flags satisfying the first condition (we consider $H$ as a point of the dual projective space which is not  in 
the union of the hyperplanes corresponding to $p,p_1,p_2$).
The second condition  holds for precisely
$$((q-1)q^{n-2} -(q-1))\cdot q^{n-3}=(q^{n-2}-1)(q-1)q^{n-3}$$
anti-flags and for each $i\in \{1,2\}$ there are precisely 
$$(q^{n-2}-1) \cdot (q-1)q^{n-3}$$
anti-flags satisfying the third condition. 
We get the claim.

(2). 
Suppose now that $A_1\sim_3 A_2$.
By duality, we assume that $p_1=p_2$. 
Then for every anti-flag $(p,H)$ belonging to $\{A_1,A_2\}^{\sim_1}$ one of the following possibilities is realized:
\begin{enumerate}
\item[$\bullet$] $p\in H_1\cap H_2$ and $p_1=p_2\not\in H$;
\item[$\bullet$] $p\not\in H_1\cup H_2\cup \{p_1=p_2\}$ and $p_1=p_2\in H$.
\end{enumerate}
There are precisely 
$$\frac{q^{n-2}-1}{q-1}\cdot (q-1)q^{n-2}=(q^{n-2}-1)q^{n-2}
\;\mbox{ and }\;
((q-1)q^{n-2}-1)\cdot q^{n-2}$$
anti-flags satisfying the first and second condition, respectively.
The sum of these numbers is $q^{2n-3} - 2q^{n-2}$.

(3) and (4). If $A_1\sim_4 A_2$, then  $\{p_1,p_2\}\cap (H_1\cup H_2)=\emptyset$
and for every anti-flag $(p,H)$ belonging to $\{A_1,A_2\}^{\sim_1}$ one of the following possibilities is realized:
\begin{enumerate}
\item[$\bullet$] $p\in H_1\cap H_2$ and $H$ intersects $\langle p_1, p_2\rangle$ in a point distinct from $p_1$ and $p_2$;
\item[$\bullet$] $p\not\in H_1\cup H_2\cup \langle p_1, p_2\rangle$ and $\langle p_1, p_2\rangle \subset H$;
\item[$\bullet$] for one of $i\in \{1,2\}$ we have $p\in H_i\setminus H_{3-i}$, $p_i\not\in H$ and $p_{3-i}\in H$.
\end{enumerate}
Suppose first that $\langle p_1, p_2 \rangle$ intersects $H_1\cap H_2$ in a point $p'$.
There are precisely 
$$\left (\frac{q^{n-2}-1}{q-1}-1\right )\cdot (q-1)^2q^{n-3}$$
anti-flags $(p,H)$ with $p\ne p'$  and $(q-2)q^{n-2}$ anti-flags $(p',H)$
which satisfy the first condition; the sum of these numbers is 
$$(q^{n-2}-1)(q-1)q^{n-3}-q^{n-3}.$$
The second condition  holds for precisely
$$((q-1)q^{n-2}-q) \cdot q^{n-3}$$
anti-flags and for each $i\in \{1,2\}$ there are precisely 
$$q^{n-3}\cdot (q-1)q^{n-2}$$
anti-flags satisfying the third condition. 
A direct calculation shows that the sum 
$$(q^{n-2}-1)(q-1)q^{n-3}-q^{n-3}+((q-1)q^{n-2}-q)q^{n-3}+2(q-1)q^{2n-5}$$
is as required.

Suppose now that $\langle p_1, p_2 \rangle$ intersects $H_1$ and $H_2$ in distinct points $p'_1$ and $p'_2$, respectively. 
As in the case when $A_1\sim_2 A_2$, the number of anti-flags satisfying the first or second condition is $$(q^{n-2}-1)(q-1)q^{n-3}.$$
Also, for each $i\in \{1,2\}$ we have the same number of  anti-flags $(p,H)$ with $p\ne p'_i$ satisfying the third condition. 
Since the assumption that $p_{3-i}\in H$ and $p_i\not\in H$ implies that $p'_i\not \in H$, 
the number of anti-flags $(p'_i,H)$ satisfying the third condition is $q^{n-2}$. 
We obtain the required equality.
\end{proof}

\section{Non-singular points of the hyperbolic form over the field of two elements as anti-flags}
As above, we suppose that $V={\mathbb F}^n$ and $n\ge 3$.
If $x\in V$ and $x^*\in V^*$, then $x^{*}x$ denotes the value of the functional $x^*$ on the vector $x$.
Recall that for non-zero $x^*\in V^*$ the annihilator $(x^*)^0$  is the $(n-1)$-dimensional subspace of $V$ formed by all $x\in V$ satisfying $x^*x=0$.

Consider the non-degenerate quadratic form $Q$ on the $2n$-dimensional vector space $V\times V^*$ defined as follows:
$$Q({\bf x})=x^*x\;\mbox{ for }\; {\bf x}=(x,x^*)\in V\times V^*.$$
It is clear that 
$$\Pi=\{(x,0):x\in V\}\;\mbox{ and }\;\Sigma=\{(0,x^*):x^*\in V^*\}$$
are maximal singular subspaces of $Q$.
Also, $(x,x^*)\not\in \Pi\cup \Sigma$  is a non-singular vector of $Q$ if and only if $\langle x \rangle$ and $(x^*)^0$   form an anti-flag.

Singular points of ${\rm PG}(2n-1, {\mathbb F})$, i.e. corresponding to singular $1$-dimensional subspaces of $Q$,
form the hyperbolic polar space ${\mathcal O}^{+}(2n, {\mathbb F})$ (we write  ${\mathcal O}^{+}(2n, q)$ in the case when $|{\mathbb F}|=q$).
The remaining points of ${\rm PG}(2n-1, {\mathbb F})$ are said to be {\it non-singular}.

Let $f$ be the map of the set of non-singular points of ${\rm PG}(2n-1, {\mathbb F})$ to the set of anti-flags ${\mathcal A}$
sending every non-singular point $\langle (x,x^*)\rangle$ to the anti-flag formed by $\langle x \rangle$ and $(x^*)^0$.
This map is surjective, but it is not injective if $|{\mathbb F}|\ge 3$.
Indeed, if $|{\mathbb F}|\ge 3$, ${\bf p}=\langle (x,x^*)\rangle$ is a non-singular point and $a\in {\mathbb F}\setminus \{0,1\}$,
then ${\bf p}'=\langle (x,ax^*)\rangle$ is a non-singular point distinct from ${\bf p}$ and $f({\bf p})=f({\bf p}')$.

So, $f$ is bijective if and only if  $|{\mathbb F}|=2$. 

From now on we assume that $|\mathbb{F}| = 2$. The next proposition is the main result of \cite{IP}. The setting we have introduced above allows us to give a very short proof of that result.

\begin{prop}\label{prop-F2-2}
Suppose that $|{\mathbb F}|=2$. For distinct non-singular points ${\bf p}_1, {\bf p}_2$ the third point on the line joining them 
is non-singular if and only if the anti-flags $f({\bf p}_1),f({\bf p}_2)$ are $1$-adjacent.
\end{prop}

\begin{proof}
Let ${\bf p}_i=\langle (x_i, x^*_i) \rangle$, $i\in \{1,2\}$ be distinct non-singular points. 
Then $x^*_ix_i=1$ for every $i$. The third point on the line $\langle {\bf p}_1,{\bf p}_2\rangle$ is 
$\langle (x_1+x_2,x^*_1+x^*_2)\rangle$
and
$$(x^*_1+x^*_2)(x_1+x_2)=x^*_1x_1+x^*_1x_2+x^*_2x_1+x^*_2x_2=x^*_1x_2+x^*_2x_1$$
is non-zero if and only if precisely one of $x^*_1x_2, x^*_2x_1$ is non-zero.
This gives the claim.
\end{proof}

Using the map $f$, we identify non-singular points of the form $Q$ with anti-flags of ${\rm PG}(n-1,2)$.
By Proposition \ref{prop-F2-2}, $\Gamma_1$ is the collinearity graph of the point-line geometry of non-singular points and totally non-singular lines 
(lines formed by non-singular points).  

\begin{theorem}\label{th2}
The polar space ${\mathcal O}^{+}(2n, 2)$ can be recovered from the graph $\Gamma_1$. 
\end{theorem}

The automorphism group of the polar space ${\mathcal O}^{+}(2n, 2)$ is the orthogonal group ${\rm O}^{+}(2n,2)$. 
Every element of this group induces an automorphism of  the graph $\Gamma_1$. 
In particular, elements of ${\rm O}^{+}(2n,2)$ preserving each of the singular subspaces $\Pi$ and $\Sigma$ correspond to linear automorphisms of $V$, 
elements transposing these subspaces are related to linear isomorphisms of $V$ to $V^{*}$.

Conversely, Theorem \ref{th2} shows that every automorphism of $\Gamma_1$ induces an automorphism of the polar space ${\mathcal O}^{+}(2n, 2)$
which implies the following. 

\begin{cor}\label{cor2}
If $|{\mathbb F}|=2$, then the automorphism group of $\Gamma_1$ is ${\rm O}^{+}(2n,2)$.
\end{cor}

Let $i\in \{2,3,4\}$. By the first part of Theorem \ref{th1}, the relation $\sim_1$ can be recovered from $\sim_i$. 
Therefore, if $\sim_i$ can be recovered from $\sim_1$, 
then the graphs $\Gamma_1$ and $\Gamma_i$ have the same automorphism group which contradicts
Corollaries \ref{cor1} and \ref{cor2}.
We obtain the second part of  Theorem \ref{th1}.

\section{Proof of Theorem \ref{th2}}
Consider the non-degenerate bilinear form $F$ on $V\times V^{*}$  defined as follows:
$$F({\bf x}, {\bf y})=y^*x+ x^*y\;\mbox{ for }\; {\bf x}=(x,x^*), {\bf y}=(y,y^*)\in V\times V^*.$$
The associated orthogonal relation on ${\rm PG}(2n-1,2)$ will be denoted by $\perp$. 
We use the following simple observations:
\begin{enumerate}
\item[$\bullet$] Two distinct singular points are orthogonal if and only if the line joining them is singular.
\item[$\bullet$] Two distinct non-singular points are non-orthogonal if and only if the line joining them is totally non-singular. 
\item[$\bullet$] A singular point ${\bf p}$ and a non-singular point ${\bf q}$ are orthogonal if and only if 
the third point on the line $\langle {\bf p}, {\bf q} \rangle$ is non-singular.
\end{enumerate}

Recall that non-singular points of the form $Q$ are identified with anti-flags of ${\rm PG}(n-1,2)$
and $\Gamma_1$ is the collinearity graph of the point-line geometry of non-singular points and totally non-singular lines.
So, totally non-singular lines are cliques of the graph $\Gamma_1$.
Our first step is the following characterization of such lines in therms of the relation $\sim_1$.

\begin{prop}\label{prop2-1}
Let ${\mathcal C}$ be a $3$-element clique of $\Gamma_1$.
Then ${\mathcal C}$  is a  line if and only if the following condition is satisfied:
\begin{enumerate}
\item[{\rm (NS)}] every non-singular point ${\bf p}\not\in {\mathcal C}$ is $1$-adjacent to precisely two points of ${\mathcal C}$
or it is not $1$-adjacent to any point of ${\mathcal C}$.
\end{enumerate}
\end{prop}

This shows that every totally non-singular line is a maximal clique of $\Gamma_1$.

\begin{proof}[Proof of Proposition \ref{prop2-1}]
Suppose that ${\mathcal C}=\{{\bf p}_1,{\bf p}_2,{\bf p}_3\}$ is a totally non-singular line.
Then 
${\bf p}^{\perp}_1\cap {\bf p}^{\perp}_2 \subset {\bf p}^{\perp}_3$
and the projective space ${\rm PG}(2n-1,2)$ is the union 
of the hyperplanes ${\bf p}^{\perp}_1,{\bf p}^{\perp}_2,{\bf p}^{\perp}_3$.
So, a non-singular point ${\bf p}\not\in {\mathcal C}$ is contained in precisely one of these hyperplanes
or in their intersection. This implies (NS); in particular, every totally non-singular line is a maximal clique of $\Gamma_1$.

Show that (NS) fails if ${\mathcal C}=\{{\bf p}_1,{\bf p}_2,{\bf p}_3\}$ spans a plane $A$. 
There is a solid $S$ containing $A$ and such that the restriction of $F$ to the  $4$-dimensional subspace of $V\times V^*$ corresponding to $S$  is non-degenerate. 
We denote this restriction by $F'$.
For every plane $X\subset S$ there is a unique point ${\bf p}_X\in X$ such that $S\cap {\bf p}^{\perp}_X=X$.

Since the line $\langle {\bf p}_i,{\bf p}_j\rangle$ is totally non-singular for any distinct $i,j\in \{1,2,3\}$,
all points of $A$ are non-singular except ${\bf p}_A$.
Let $B$ and $C$ be the remaining two planes of $S$ passing through the line $\langle {\bf p}_2,{\bf p}_3\rangle$. 
If one of the points ${\bf p}_B,{\bf p}_C$ is non-singular,  say ${\bf p}_B$, then ${\bf p}_B$ and ${\bf p}_1$ are $1$-adjacent 
(these points are non-orthogonal, since $F'$ is non-degenerate). 
On the other hand, ${\bf p}_B$ is not $1$-adjacent to ${\bf p}_2,{\bf p}_3$ which contradicts (NS).

Suppose now that ${\bf p}_B,{\bf p}_C$ are singular. 
Since ${\bf p}_B$ is orthogonal to ${\bf p}_2,{\bf p}_3$ and $\langle {\bf p}_2,{\bf p}_3\rangle$ is totally non-singular,
every point of $B$ distinct from ${\bf p}_B$ is non-singular. 
The same holds for the plane $C$. 
Therefore, all points of $S$ are non-singular except ${\bf p}_A,{\bf p}_B,{\bf p}_C$. 
These points are mutually non-orthogonal (since $F'$ is non-degenerate) and, consequently, 
they span a plane $D$. Then there is a plane of $S$ intersecting $D$ in a totally non-singular line. 
This plane is formed by non-singular points which is impossible (every totally non-singular line is a maximal clique of $\Gamma_1$). 
So, at least one of the points ${\bf p}_B,{\bf p}_C$ is non-singular which implies that (NS) fails. 
\end{proof}

For every $2$-element coclique $\{{\bf p},{\bf q}\}$ of $\Gamma_1$ the third point on the line $\langle {\bf p},{\bf q}\rangle$ is singular. 
Two distinct $2$-element cocliques of $\Gamma_1$ are said to be {\it parallel} if the corresponding lines are intersecting in a singular point.
Every singular point is determined by a class of parallel $2$-element cocliques of $\Gamma_1$.
So, we need to characterize such classes in terms of the relation $\sim_1$.

For any two parallel $2$-element cocliques of $\Gamma_1$
the corresponding lines are intersecting in a singular point ${\bf p}$ and span a plane contained in ${\bf p}^{\perp}$.
There are precisely two possibilities:
\begin{enumerate}
\item[(1)] all points on the plane except ${\bf p}$ are non-singular,
\item[(2)]  singular points of the plane form a line. 
\end{enumerate}
We say that these are parallel cocliques of {\it the first} or {\it the second type}
if the corresponding possibility is realized. 

Proposition \ref{prop2-1} provides a characterization of  totally non-singular lines  in terms of $\sim_1$. 
Therefore, parallel cocliques of the first type can be characterized in terms of $\sim_1$  as follows.

\begin{lemma}\label{lemma2-1}
Let ${\mathcal C}_1,{\mathcal C}_2$ be distinct $2$-element cocliques of $\Gamma_1$.
These are parallel cocliques of the first type if and only if  there are two distinct  totally non-singular lines which intersect each of these cocliques
and are intersecting outside of ${\mathcal C}_1\cup {\mathcal C}_2$.
\end{lemma}

\begin{proof}
Easy verification. 
\end{proof}

\begin{lemma}\label{lemma2-2}
Let ${\mathcal C}_1,{\mathcal C}_2$ be distinct $2$-element cocliques of $\Gamma_1$ which are not parallel cocliques of the first type. 
These are parallel cocliques of the second type if and only if  
there is a $2$-element coclique ${\mathcal C}'$ such that ${\mathcal C}_i,{\mathcal C}'$ are parallel cocliques of the first type for every $i\in \{1,2\}$. 
\end{lemma}

\begin{proof}
Let $\ell_1$ and $\ell_2$ be the lines containing ${\mathcal C}_1$ and ${\mathcal C}_2$, respectively. 
If there is ${\mathcal C}'$ such that ${\mathcal C}_i,{\mathcal C}'$ are parallel cocliques of the first type for every $i\in \{1,2\}$,
then $\ell_1$ and $\ell_2$ are intersecting in a singular point which means that ${\mathcal C}_1,{\mathcal C}_2$ are parallel. 

Suppose now that ${\mathcal C}_1,{\mathcal C}_2$ are parallel cocliques of the second type
and the lines $\ell_1,\ell_2$ are intersecting in a singular point ${\bf p}$.
We take non-singular points ${\bf p}_1\in \ell_1$ and ${\bf p}_2\in \ell_2$. 
These points are orthogonal (since ${\mathcal C}_1,{\mathcal C}_2$ are parallel cocliques of the second type). 
The line $\langle {\bf p}_1,{\bf p}_2\rangle$ does not pass through ${\bf p}$ which guarantees that 
the intersections of ${\bf p}^{\perp}$ with ${\bf p}^{\perp}_1$ and ${\bf p}^{\perp}_2$ are distinct. 
It is sufficient to show that there is a non-singular point 
$${\bf q}\in {\bf p}^{\perp}\setminus ({\bf p}^{\perp}_1\cup {\bf p}^{\perp}_2).$$ 
Indeed, if ${\bf q}'\in \langle {\bf p}, {\bf q}\rangle$ is distinct from ${\bf p}$ and ${\bf q}$,
then ${\bf q}'$ is non-singular and the coclique ${\mathcal C}'=\{{\bf q}, {\bf q}'\}$ is as required.

If $n>3$, then there is a $6$-dimensional subspace of $V\times V^*$ containing $\ell_1\cup \ell_2$ 
and such the restriction of $F$ to this subspace is non-degenerate of Witt index equal to 3. 
For this reason we can assume with no loss that $n=3$.
Then ${\bf p}^{\perp}\cap {\bf p}^{\perp}_1\cap {\bf p}^{\perp}_2$ is the plane spanned by $\ell_1\cup \ell_2$;
we denote this plane by $S$.
For every $i\in \{1,2\}$ there are precisely four non-singular points  in $({\bf p}^{\perp}\cap {\bf p}^{\perp}_i)\setminus {\bf p}^{\perp}_{3-i}$
(the line joining ${\bf p}_i$ with any singular point from $({\bf p}^{\perp}\cap {\bf p}^{\perp}_i)\setminus {\bf p}^{\perp}_{3-i}$ contains a non-singular point).
These four non-singular points are lying on two lines passing through ${\bf p}$. 
Let $S_i$ be the plane spanned by these lines.
Then $S\cap S_i$ is the singular line of $S$ passing through ${\bf p}$. 
Therefore, ${\bf p}^{\perp}\cap {\bf p}^{\perp}_i$ is spanned by $S_i$ and $\ell_{3-i}$. 
We take any non-singular point 
$${\bf q}_1\in ({\bf p}^{\perp}\cap {\bf p}^{\perp}_1)\setminus {\bf p}^{\perp}_2.$$
Then $\ell_1\subset {\bf q}^{\perp}_1$ and, consequently, ${\bf q}^{\perp}_1$ does not contain $S_2$ 
(otherwise ${\bf p}^{\perp}\cap{\bf q}^{\perp}_1$ coincides with ${\bf p}^{\perp}\cap {\bf p}^{\perp}_2$ which is impossible, since ${\bf p}\not\in \langle {\bf q}_1,{\bf p}_2 \rangle$). 
So, ${\bf q}^{\perp}_1$ intersects $S_2$ in a line 
which implies the existence of a non-singular point 
$${\bf q}_2 \in ({\bf p}^{\perp}\cap {\bf p}^{\perp}_2)\setminus {\bf p}^{\perp}_1$$
non-orthogonal to ${\bf q}_1$. 
The non-singular point ${\bf q}\in \langle {\bf q}_1,{\bf q}_2\rangle$ distinct from ${\bf q}_1,{\bf q}_2$ is as required.
\end{proof}

Combining Lemmas \ref{lemma2-1}  and   \ref{lemma2-2}  
we obtain a characterization of parallel classes of $2$-element cocliques of $\Gamma_1$ in terms of the relation $\sim_1$. 
So, the set of points of the polar space  ${\mathcal O}^{+}(2n, 2)$ can be recovered from the graph $\Gamma_1$.
Singular lines are recovered from $\Gamma_1$ as follows.

\begin{prop}
Let ${\bf p}_1,{\bf p}_2, {\bf p}_3$ be mutually distinct singular points 
and let ${\mathfrak C}_i$, $i\in \{1,2,3\}$ be the parallel class of $2$-element  cocliques corresponding to ${\bf p}_i$.
These points form a line if and only if there are 
cocliques ${\mathcal C}_1\in {\mathfrak C}_1, {\mathcal C}_2\in {\mathfrak C}_2, {\mathcal C}_3\in {\mathfrak C}_3$
such that $C_i\cap C_j$ is a point for every choice of distinct indices $i, j\in \{1,2,3\}$ and the points $C_1\cap C_2$, $C_2\cap C_3$ and $C_3\cap C_1$ are mutually distinct. 
\end{prop}

\begin{proof}
If $\ell=\{{\bf p}_1,{\bf p}_2, {\bf p}_3\}$ is a line,
then $\ell^{\perp}$ is a $(2n-3)$-dimensional subspace of ${\rm PG}(2n-1,2)$ 
and, consequently, it contains a non-singular point ${\bf p}$ (since $n\ge 3$).
There are precisely four non-singular points on the plane spanned by $\ell$ and ${\bf p}$.
This plane contains three cocliques satisfying the required condition.

Conversely, suppose that ${\mathcal C}_1\in {\mathfrak C}_1, {\mathcal C}_2\in {\mathfrak C}_2, {\mathcal C}_3\in {\mathfrak C}_3$
are mutually intersecting in three distinct points. 
Then $\ell_i={\mathcal C}_i\cup\{{\bf p}_i\}$ is a line for every $i\in \{1,2,3\}$
and these lines span a plane. 
Therefore,  the third point on the line $\langle {\bf p}_1,{\bf p}_2\rangle$ belongs to $\ell_3\setminus(\ell_1\cup \ell_2)$
and, consequently, this point is ${\bf p}_3$.
\end{proof}


\begin{thebibliography}{9}
%\bibitem{vBok} M. M. van Bokhoven, {\it Automorphism groups of tangent graphs to polar spaces}, bachelor thesis, TU$/$e, 2024. 
\bibitem{BvM}  A. E. Brouwer, H. Van Maldeghem,
{\it Strongly regular graphs},
Encyclopedia of Mathematics and its Applications 182, 
Cambridge University Press 2022.

\bibitem{CCG} A. M. Cohen, H. Cuypers, R. Gramlich, {\it Local recognition of non-incident point-hyperplane graphs}, Combinatorica 25 (2005), 271-296. 

\bibitem{Die}
J. Dieudonn\'e,  {\it La G\'eom\'etrie des Groupes Classiques},  3d edn., Springer 1971.

\bibitem{Mackey}
G. W.  Mackey, {\it Isomorphisms of normed linear spaces}, Ann. of Math. 43(1942), 244-260.

%\bibitem{HP} H. Havlicek, M. Pankov, 
%{\it Transformations on the product of Grassmann spaces}, Demonstr. Math. 38(2005), 675-688. 

\bibitem{Pank-book}
M. Pankov, {\it Grassmannians of classical buildings}, World Scientific 2010.

\bibitem{IP} F. Ihringer, A. Pasini, {\it Bijection between point-hyperplane anti-flags of $V(n, 2)$ and non-singular points of  $O^+(2n, 2)$},
arXiv: 2509.14798 (2025).

\end{thebibliography}
\end{document}